\documentclass{amsart}

\usepackage{color}
\usepackage{soul}
\usepackage[pagewise]{lineno}

\usepackage{hyperref}
\hypersetup{
    colorlinks=true, 
    linktoc=all,     
    linkcolor=blue,  
}

\newcommand{\Q}{\mathbb{Q}}
\newcommand{\C}{\mathbb{C}}

\newcommand{\N}{\mathbb{N}}

\newcommand{\norm}[1]{\left\| #1 \right\|}
\newcommand{\abs}[1]{\left\vert #1 \right\vert}

\newcommand{\boldA}{\mathbf{A}}

\newcommand{\cstar}{$\mathrm{C}^*$}

\newcommand{\calH}{\mathcal{H}}
\newcommand{\calK}{\mathcal{K}}

\newcommand{\calS}{\mathcal{S}}

\theoremstyle{plain}
\newtheorem{theorem}{Theorem}[section]
\newtheorem{lemma}[theorem]{Lemma}
\newtheorem{corollary}[theorem]{Corollary}
\newtheorem{proposition}[theorem]{Proposition}

\newtheorem{question}[theorem]{Question}
\newtheorem{fact}[theorem]{Fact}

\theoremstyle{definition}
\newtheorem{definition}[theorem]{Definition}

\numberwithin{equation}{section}

\begin{document}
\title{The computable functional calculus}
\author[C. J. Eagle]{Christopher J. Eagle${}^1$} 
\address[C. J. Eagle]
{Department of Mathematics and Statistics, University of Victoria. PO BOX 1700 STN CSC, Victoria, British Columbia, Canada. V8W 2Y2}%
\email{eaglec@uvic.ca}
\urladdr{http://web.uvic.ca/~eaglec}
\thanks{${}^1$ Supported by NSERC Discovery Grant RGPIN-2021-02459}

\author[T. H. McNicholl]{Timothy H. McNicholl}
\address[T. H. McNicholl]{Iowa State University}
\email{mcnichol@iastate.edu}
\urladdr{https://faculty.sites.iastate.edu/mcnichol/}

\begin{abstract}
We show that the continuous functional calculus is computable.  As consequences we obtain the computable compactness of the spectrum of any computable normal element of a computably presented \cstar-algebra, the existence of effective approximate units for computably presented \cstar-algebras, and an effective version of the Spectral Theorem for compact operators on separable Hilbert spaces.
\end{abstract}
\maketitle


\section{Introduction}\label{sec:Intro}
Given an element $a$ of a \cstar-algebra $\boldA$ and a polynomial $f$, there is a natural way of producing another element $f(a) \in \boldA$.  The continuous functional calculus, which is a fundamental tool in the study of \cstar-algebras, extends this to the case where $f$ is no longer required to be a polynomial, but just a function that is continuous on the spectrum of $a$, provided that $a$ is a normal element.

Following work of Fox \cite{fox2022computable}, there has recently been considerable activity in the study of computable \cstar-algebras (for example, \cite{Gelfand, AFPaper, UHFPaper, NonComputabilityPaper, FoxGoldbringHart.2024, Goldbring.2026, McNicholl.2025}).  In some of the arguments in these papers there has been a need for a computable version of the continuous functional calculus; that is, given a computable normal point $a$ of a presentation $\boldA^\#$ of $\boldA$, and a computable function $f$ on the spectrum of $a$, one frequently wants to know that $f(a)$ is also a computable point of $\boldA^\#$.  So far, these situations have been handled in an ad-hoc way, for example by using a power series representation for $f$ in the case where $f$ is analytic.  In this paper we give a systematic treatment of the continuous functional calculus in a computability setting, which provides a unified way to handle arguments of this kind.  Specifically, in Theorem \ref{theorem:CompFuncCalc} and Theorem \ref{theorem:ComputableCompactSpectrum} we prove:

\begin{theorem}
Let $\boldA^\#$ be a computable presentation of a unital \cstar-algebra $\boldA$, and suppose that the unit of $\boldA$ is a computable point of $\boldA^\#$.  
Let $a$ be a normal computable point of $\boldA^\#$.  Then, $\sigma(a)$, the spectrum of $a$, is a computably compact subset of $\mathbb{C}$, and the continuous functional calculus is a computable map from $C(\sigma(a))$ to $\boldA^\#$.
\end{theorem}

In particular, this shows that the application of a computable function on $\sigma(a)$ to $a$ produces a computable point of $\boldA^\#$ (Corollary \ref{corollary:ApplyingFuncCalcUnital}).  We also extend our results to the non-unital setting.  Some of our results generalize work of Brattka and Dillhage \cite{BrattkaDillhage.2005} and of Dillhage \cite{Dillhage.2008} from the case of a single operator on a Hilbert space to the setting of \cstar-algebras.  As a sample application, we show that if $\boldA^\#$ is a computable presentation of a \cstar-algebra $\boldA$, then there is a computable sequence of $\boldA^\#$ that is an effective approximate unit for $\boldA^\#$ (Theorem \ref{thm:ApproximateUnit}).

In Section \ref{section:CompactOperators}, we study compact operators on Hilbert spaces.  Classically, compact operators on Hilbert spaces are equivalently characterized as either those operators where the image of the closed unit ball is pre-compact or those operators that can be uniformly approximated by finite-rank operators.  We describe natural computability-theoretic versions of both properties and prove (Theorem \ref{theorem:EquivalentCompactness}) that the two are equivalent.  We conclude by applying our computable functional calculus to the \cstar-algebra of compact operators to obtain an effective version of the Spectral Theorem (Theorem \ref{thm:Spectral}).

\textbf{Acknowledgement}: We thank an anonymous referee for a detailed and helpful report.
\section{Preliminaries}

Throughout this paper, all \cstar-algebras are assumed to be separable.  The setting we use to study the computability of \cstar-algebras, Hilbert spaces, and related structures is the one presented in \cite{EMST}, and we refer the reader there for a detailed introduction.  The necessary background can also be found in \cite{UHFPaper}.

When we consider the \cstar-algebra $\mathbb{C}$, we always view it as having its standard presentation where the only special point is $1$ (and thus the generated points are the points of $\mathbb{Q}(i)$).  We identify $\mathbb{C}$ with this presentation throughout the paper.  To avoid repetition, \emph{throughout the paper $\boldA$ is a \cstar-algebra and $\boldA^\#$ is a fixed computable presentation of $\boldA$}.  

Our \cstar-algebra $\boldA$ may or may not be unital.  Even if it is unital, there is no guarantee that the unit will be a computable point of $\boldA^\#$.  This leads us to 
the following definition.

\begin{definition}
We say that a presentation $\boldA^\#$ of a unital \cstar-algebra
is \emph{computably unital} if the $1_\boldA$ is a computable point of $\boldA^\#$.
\end{definition}

At present, we do not know of an example of a computable presentation of a unital 
\cstar-algebra that is not computably unital.  It is known (see \cite[Theorem 4.1]{Gelfand}), that every computable presentation of a commutative unital \cstar-algebra
is computably unital.  More generally, every presentation of a stably finite \cstar-algebra is computably unital \cite[Proposition 4.5]{UHFPaper}.  However, 
there is a subtlety here regarding uniformity; namely, it is not possible to 
compute, from an index of a 
computably unital presentation $\boldA^\#$, an $\boldA^\#$-index of $1_\boldA$ (see 
\cite[Theorem 5.3]{McNicholl.2025}).  These considerations motivate the following 
definition; here $\langle, \rangle$ denotes the standard bijection of $\N^2$ onto $\N$.

\begin{definition}
Let $\boldA^\#$ be a computable presentation of a unital \cstar-algebra, and 
assume $\boldA^\#$ is computably unital.  Then, $e \in \N$ is a \emph{unital index} of 
$\boldA^\#$ if $e = \langle e_0, e_1 \rangle$ where $e_0$ is an index of $\boldA^\#$
and $e_1$ is a $\boldA^\#$-index of $1_\boldA$. 
\end{definition}

Recall (see, e.g., \cite[Definition 1.1.14]{Emerson}) that the (external) \emph{unitization} of a \cstar-algebra $\boldA$ (where $\boldA$ may or may not already be unital) is the unique \cstar-algebra $\boldA_1$ such that $\boldA$ is a closed two-sided ideal of $\boldA_1$ and $\boldA_1/\boldA \cong \mathbb{C}$.  Concretely, $\boldA_1 = \boldA \oplus \mathbb{C}$ as a vector space, with multiplication $(a, \lambda)\cdot(b, \mu) = (ab+\lambda b + \mu a, \lambda\mu)$ and involution $(a, \lambda)^* = (a^*, \overline{\lambda})$.  The unique norm making this *-algebra into a \cstar-algebra is given by
\[\norm{(a, \lambda)}_{\boldA_1} = \max\{\abs{\lambda}, \sup\{\norm{ab+\lambda b}_{\boldA} : b \in \boldA, \norm{b}_{\boldA} \leq 1\}\}.\]
The unit of $\boldA_1$ is $(0, 1)$.

Now, suppose $\boldA^\#$ is a presentation of a \cstar-algebra.  
This presentation naturally induces a presentation $(\boldA^\#)_1$ of 
$\boldA_1$ in which the special points are those of the form 
$(p, \lambda)$ where $p$ is a special point of $\boldA^\#$ and 
$\lambda \in \Q(i)$.  We note that the inclusion map is a 
computable map of $\boldA^\#$ to $(\boldA^\#)_1$ and that 
$(\boldA^\#)_1$ is computably unital.
Moreover, we have the following which was proved in \cite[Proposition 2.21]{UHFPaper}.

\begin{fact}[Computable unitization]\label{prop:CompUnit}
If $\boldA^\#$ is a computable presentation of a \cstar-algebra, 
then $(\boldA^\#)_1$ is also computable.  Furthermore, a unital index 
of $(\boldA^\#)_1$ can be computed from an index of $\boldA^\#$.
\end{fact}

\begin{definition}
Suppose that $\boldA$ is unital. The \emph{spectrum} of an element $a \in \boldA$ is
\[\sigma_\boldA(a) = \{\lambda \in \mathbb{C} : a - \lambda1_\boldA \text{ is not invertible}\}.\]
If $\boldA$ is non-unital, then by definition $\sigma_\boldA(a) = \sigma_{\boldA_1}(a)$.
\end{definition}

When the \cstar-algebra $\boldA$ is clear from context we omit it from the notation for the spectrum.  Note, however, that if $\boldA$ is unital then it may be the case that $\sigma_{\boldA}(a) \neq \sigma_{\boldA_1}(a)$, because the unit of $\boldA$ is not the same as the unit of $\boldA_1$.  In fact, if $a \in \boldA$ then $a$ cannot be invertible in $\boldA_1$ and so $0 \in \sigma_{\boldA_1}(a)$ (regardless of whether or not $0 \in \sigma_\boldA(a)$).  This is the only possible difference between the spectra, so the relationship between $\sigma_{\boldA}(a)$ and $\sigma_{\boldA_1}(a)$ when $\boldA$ is unital is that $\sigma_{\boldA_1}(a) = \sigma_{\boldA}(a) \cup \{0\}$; see \cite[paragraph after Remark 2.2.1]{Murphy.1990}

In Section \ref{section:CompactOperators} we will consider operators on Hilbert spaces.  The following well-known fact was shown, in slightly different terminology, in \cite[Lemma 7]{PourElRichards.1989}; for the version in the terminology we use in this paper, see \cite[Lemma 3.1]{BrattkaYoshikawa}.

\begin{fact}\label{fact:ONB}
Every computable presentation of a Hilbert space has a computable orthonormal basis.
Specifically:
\begin{enumerate}
    \item If $\calH^\#$ is a computable presentation of an $n$-dimensional Hilbert space, 
    then there exist computable vectors $e_0, \ldots, e_{n-1}$ of 
    $\calH^\#$ so that $(e_j)_{j < n}$ is an orthonormal basis of $\calH$. 

    \item If $\calH^\#$ is a computable presentation of an 
    infinite-dimensional Hilbert space, 
    then there exists a computable sequence $(e_j)_{j \in \N}$ of 
    $\calH^\#$ so that $(e_j)_{j \in \N}$ is an orthonormal basis of $\calH$. 
\end{enumerate}
\end{fact}

\section{Continuous functional calculus}
Recall that an element of $\boldA$ is \emph{normal} if $aa^* = a^*a$.  For normal elements of \cstar-algebras we have the continuous functional calculus:

\begin{fact}[Continuous functional calculus]
Suppose that $\boldA$ is unital.  Given a normal element $a \in \boldA$, there is a unique *-isomorphism $\Phi$ of $C(\sigma(a))$ onto $C^*(1_\boldA, a)$
so that $\Phi(1) = 1_\boldA$ and $\Phi(z) = a$ ( where $1$ is the constant function $1$ and $z : \sigma(a) \to \mathbb{C}$ is the identity function). 
\end{fact}

We refer the reader to \cite[Section 2.1]{Murphy.1990} for an exposition of the continuous functional calculus and a proof of the fact above.  Given a continuous function $f : \sigma(a) \to \mathbb{C}$, we usually write $f(a)$ instead of $\Phi(f)$.

Our goal in this section is to obtain a computable version of the continuous functional calculus and some of its consequences.  We begin by discussing presentations of 
the relevant spaces.  First, if $a$ is a normal element of a unital 
\cstar-algebra $\boldA$, then the \emph{standard presentation of 
$C^*(1_\boldA, a)$} is the presentation whose only special points are 
$1_\boldA$ and $a$.  
Second, given a compact set $K \subseteq \mathbb{C}$, the \emph{standard presentation} of $C(K)$ is the presentation whose special points are the constant function $1$ and the identity function $z : K \to \mathbb{C}$.
It follows from the Weierstrass Approximation Theorem that the standard presentation of $C(K)$ is, in fact, a presentation.  Whenever we consider computability properties of $C^*(1_\boldA, a)$ or $C(\sigma(a))$, we use these standard presentations unless we explicitly state otherwise; moreover we identify these spaces with their standard 
presentations. 

We now arrive at our result on computing the functional calculus. 

\begin{theorem}\label{theorem:CompFuncCalc}
Suppose that $\boldA$ is unital, $\boldA^\#$ is computably unital, and that $a \in \boldA$ is a normal computable point of $\boldA^\#$.  Let $\Phi : C(\sigma(a)) \to C^*(1_\boldA, a)$ be the continuous functional calculus map.  Then:
\begin{enumerate}
    \item{$\Phi$ is a computable *-isomorphism from the standard presentation of $C(\sigma(a))$ to the standard presentation of $C^*(1_\boldA, a)$.}

    \item{The standard presentation of $C(\sigma(a))$ is computable.}

    \item{$\Phi$ is a computable map from the standard presentation of $C(\sigma(a))$ to $\boldA^\#$.}
\end{enumerate}
\end{theorem}

\begin{proof}
By definition of the standard presentations of $C(\sigma(a))$ and $C^*(1_\boldA, a)$, the map $\Phi$ is computable on the generated points of $C(\sigma(a))$.  Since $\Phi$ is a *-isomorphism we conclude that $\Phi$ is a computable *-isomorphism which establishes (1).

Since $\boldA^\#$ is computably unital, $1_\boldA$ is a computable point of $\boldA^\#$, and by hypothesis so is $a$.  Thus, since $\boldA^\#$ is computable, the standard presentation of $C^*(1_\boldA, a)$ is computable.  By (1), the standard presentation of $C(\sigma(a))$ is computably *-isomorphic to the standard presentation of $C^*(1_\boldA, a)$, which proves (2).

For statement (3), because $1_\boldA$ and $a$ are computable points of $\boldA^\#$, the inclusion map from $C^*(1_\boldA, a)$ to $\boldA$ is a computable map from the standard presentation of $C^*(1_\boldA, a)$ to $\boldA^\#$, and thus (3) follows from (1).
\end{proof}

\section{The computable compactness of the spectrum}

In order to use Theorem \ref{theorem:CompFuncCalc} to apply standard functional calculus arguments in an effective setting, we first need to establish the computable compactness of the spectrum of a normal element.  We will make use of the following definition from \cite{McNicholl.2025}.

\begin{definition}
Let $X$ be a compact Polish space and let $X^\#$ be a presentation of $X$.  A presentation $C(X)^\#$ of $C(X)$ is \emph{evaluative over $X^\#$} if the evaluation map $C(X) \times X \to \mathbb{C}$ given by $(f,x) \mapsto f(x)$ is a computable map from the induced presentation on $C(X) \times X$ to $\mathbb{C}$.
\end{definition}

\subsection{The unital case}

\begin{theorem}\label{theorem:ComputableCompactSpectrum}
Suppose that $\boldA$ is unital, $\boldA^\#$ is computably unital, and that $a \in \boldA$ is a normal computable point of $\boldA^\#$.  Then $\sigma(a)$ is a computably compact subset of $\mathbb{C}$.
\end{theorem}

\begin{proof}
Since the standard presentation of $C(\sigma(a))$ is computable by Theorem \ref{theorem:CompFuncCalc}, it follows from \cite[Theorem 1.1 and 1.2]{McNicholl.2025} that there is a computably compact presentation $\sigma(a)^\#$ of $\sigma(a)$ over which $C(\sigma(a))^\#$ is evaluative.  We note that the metric of $C(\sigma(a))^\#$ may 
not be the Euclidean metric.  However, the identity function $z : \sigma(a) \to \mathbb{C}$ is a special point of the standard presentation of $C(\sigma(a))$; so, since this presentation is evaluative over $\sigma(a)^\#$, the identity map is computable from $\sigma(a)^\#$ to $\mathbb{C}$.  As $\sigma(a)^\#$ is computably compact and computable compactness is preserved by computable maps (see, e.g., \cite[Lemma 3.31]{DowneyMelnikov}), we have that $\sigma(a)$ is a computably compact subset of $\mathbb{C}$.
\end{proof}

We note that the proof of Theorem \ref{theorem:ComputableCompactSpectrum} is uniform; 
namely, from a unital index of a presentation $\boldA^\#$, it is possible to 
compute an index of $\sigma(a)$ as a computably compact subset of $\C$.

We now show that the spectrum is computably compact even if the unit is 
not a computable point of the presentation; however the proof is highly nonuniform. 

\begin{theorem}\label{thm:SpctrmNtCmpUntl}
If $a$ is a normal computable point of a presentation $\boldA^\#$ of a unital \cstar-algebra, then $\sigma(a)$ is a computably compact subset of $\mathbb{C}$.
\end{theorem}

\begin{proof}
Suppose $a$ is a normal computable point of a presentation $\boldA^\#$ of a unital \cstar-algebra.  
Since the inclusion map is a computable map from $\boldA^\#$ to $(\boldA^\#)_1$, 
it follows that $a$ is a computable point of $(\boldA^\#)_1$ as well.  Hence,
by Theorem \ref{theorem:ComputableCompactSpectrum}, 
$\sigma_{\boldA_1}(a)$ is a computably compact subset of $\mathbb{C}$. 

Since $\boldA$ is unital, $\sigma_{\boldA_1}(a) = \sigma_{\boldA}(a) \cup \{0\}$ (see \cite[paragraph after Remark 2.2.1]{Murphy.1990}).  

If $a$ is non-invertible in $\boldA$, then $0 \in \sigma_\boldA(a)$, and so $\sigma_{\boldA_1}(a) = \sigma_{\boldA}(a)$.  Hence, we assume $a$ is invertible 
in $\boldA$; thus, $0 \not\in \sigma_{\boldA}(a)$.  Let $D = \operatorname{dist}(0, \sigma_\boldA(a))$, and let $q \in \mathbb{Q}$ be such that $0 < q < \frac{D}{2}$.  Given $n \in \mathbb{N}$, we can find $N \geq n$ such that $2^{-N} < q$.  Using the computable compactness of $\sigma_{\boldA_1}(a)$ we can effectively find a finite tuple of rational open balls of radius $2^{-N}$ that cover $\sigma_{\boldA_1}(a)$.  Since $2^{-N} < q < D/2$, any of these balls that contains $0$ does not contain any points of $\sigma_\boldA(a)$.  Thus removing any of these balls that contains $0$ produces a cover of $\sigma_{\boldA}(a)$.
\end{proof}

\subsection{The non-unital and not-necessarily-unital cases}

We obtain the following by combining Theorem \ref{theorem:ComputableCompactSpectrum}
and Fact \ref{prop:CompUnit}. 

\begin{corollary}\label{cor:SpctrmNtUntl}
Suppose $\boldA^\#$ is a presentation of a non-unital \cstar-algebra.
If $a$ is a normal computable point of $\boldA^\#$, then  
$\sigma(a)$ is a computably compact set of $\C$. 
\end{corollary}

The uniformity of Theorem \ref{theorem:ComputableCompactSpectrum}
and Fact \ref{prop:CompUnit} ensure that Corollary \ref{cor:SpctrmNtUntl}
is uniform as well.   

We conclude the following by combining Theorem \ref{thm:SpctrmNtCmpUntl} and 
Corollary \ref{cor:SpctrmNtUntl}; unfortunately, the proof is highly nonuniform. 

\begin{corollary}\label{cor:CmpCmpctSpctrm}
If $a$ is a computable normal point of a presentation $\boldA^\#$ of 
a \cstar-algebra, then $\sigma(a)$ is computably compact. 
\end{corollary}

Theorem \ref{theorem:CompFuncCalc} allows us to apply standard functional calculus arguments in a computability setting.  In what follows, when we discuss computable functions on $\sigma(a)$ we are thinking of $\sigma(a)$ as a computably compact subset of $\mathbb{C}$.

\begin{corollary}\label{corollary:ApplyingFuncCalcUnital}
Let $\boldA^\#$ be a computably unital presentation of a \cstar-algebra $\boldA$ and let $a$ be a normal computable point of $\boldA^\#$.  Suppose that $f : \sigma(a) \to \mathbb{C}$ is a computable function.  Then $f(a)$ is a computable point of $\boldA^\#$.
\end{corollary}

\begin{proof}
By the Effective Weierstrass Theorem \cite[Theorem I.0.6]{PourElRichards.1989}, since $f$ is a computable function with computably compact domain, $f$ is effectively uniformly approximable by polynomials in $z$ and $\overline{z}$ with coefficients in $\mathbb{Q}(i)$, which means that $f$ is a computable point of the standard presentation of $C(\sigma(a))$.  By Theorem \ref{theorem:CompFuncCalc}, the continuous functional calculus map $\Phi : C(\sigma(a)) \to \boldA$ is computable, so $f(a) = \Phi(f)$ is a computable point of $\boldA^\#$.
\end{proof}

\begin{corollary}\label{corollary:ApplyingFuncCalc}
Let $a$ be a normal computable point of $\boldA^\#$.  Suppose that $f : \sigma(a) \cup \{0\} \to \mathbb{C}$ is a computable function and that $f(0) = 0$.  Then, $f(a)$ is a computable point of $\boldA^\#$.
\end{corollary}

\begin{proof}
By Corollary \ref{corollary:ApplyingFuncCalcUnital}, $f(a)$ is a computable point of $\boldA_1^\#$, and since $f(0) = 0$, $f(a) \in \boldA$.  Since the inclusion map from $\boldA$ to $\boldA_1$ is an isometry, it follows that $f(a)$ is also a computable point of $\boldA^\#$.
\end{proof}

For later reference, we record some specific instances of Corollary \ref{corollary:ApplyingFuncCalc}.

\begin{corollary}\label{corollary:PositiveNegativeParts}
If $a \in \boldA$ is a normal computable point of $\boldA^\#$ then the real and imaginary parts $a_r$ and $a_i$ of $a$ are computable points of $\boldA^\#$.  If $a$ is a self-adjoint computable point of $\boldA^\#$ then the positive and negative parts $a^+$ and $a^-$ are computable points of $\boldA^\#$.
\end{corollary}

\begin{proof}
The function $\operatorname{Re}(z)$ is computable on any computably compact subset of $\mathbb{C}$, and $\operatorname{Re}(0) = 0$, so we may apply Corollary \ref{corollary:ApplyingFuncCalc} to conclude that $a_r = \operatorname{Re}(a)$ is a computable point of $\boldA^\#$.  By similar reasoning, $a_i = \operatorname{Im}(a)$ is also a computable point of $\boldA^\#$.

If $a$ is positive then $\sigma(a) \subseteq [0, \infty)$.  The function $\sqrt{z}$ is computable on any computably compact subset of $[0, \infty)$, and $\sqrt{0} = 0$, so again we may apply Corollary \ref{corollary:ApplyingFuncCalc} to conclude that $a^{1/2}$ is a computable point of $\boldA^\#$.  

The remaining claims follow because $a^2$ is positive for any self-adjoint $a$ and we have $\abs{a} = (a^2)^{1/2}$, $a^+ = \frac{1}{2}(\abs{a}+a)$, and $a^- = \frac{1}{2}(\abs{a}-a)$.
\end{proof}

We note that the Corollaries above are uniform for non-unital algebras and also are uniform for computably unital presentations when given a unital index.

\section{Effective approximate units}

In many treatments of the theory of \cstar-algebras one of the early uses of the functional calculus is to show that every separable \cstar-algebra has a \emph{sequential approximate unit}.  As an example of how the results of the previous section can be applied, we show that computable \cstar-algebras have computable sequential approximate units.

\begin{definition}
A \emph{sequential approximate unit} for $\boldA$ is a sequence $(u_n)_{n \in \mathbb{N}}$ of positive elements of $\boldA$ such that for every $a \in \boldA$,
\[\lim_{n\to\infty}u_na = \lim_{n\to\infty}au_n = a.\]
\end{definition}

If $\boldA$ is unital then the constant sequence $u_n = 1_\boldA$ is an approximate unit, so the interesting case is when $\boldA$ is non-unital.  Therefore, \emph{for this section, $\boldA$ is assumed to be non-unital}. 

\begin{definition}
A sequential approximate unit is an \emph{effective approximate unit} of $\boldA^\#$ if it is a computable sequence and from a (code of a) generated point $a$ of $\boldA^\#$ it is possible to compute a modulus of convergence for $(u_na)_{n \in \mathbb{N}}$ and a modulus of convergence for $(au_n)_{n\in\mathbb{N}}$.
\end{definition}

Recall that an element $a \in \boldA$ is \emph{positive}, written $a \geq 0$, if $a$ is normal and $\sigma(a) \subseteq [0, \infty)$.  It is well-known that the positive elements of a \cstar-algebra can be equivalently described as those $a$ for which there exists a $b$ so that $a = b^*b$.  It follows, in particular, that positive elements are self-adjoint.  We let $\boldA^+$ denote the set of positive elements of $\boldA$.  The notion of positivity induces a translation-invariant partial order on the self-adjoint elements of $\boldA$ defined by $a \leq b$ if and only if $b - a \geq 0$.  See \cite[Section 2.2]{Murphy.1990} for a full discussion of positive elements of \cstar-algebras.

For the remainder of the section, let $\Lambda = \{a \in \boldA^+ : \norm{a} < 1\}$.  Classically, $\Lambda$ forms a net that is an approximate unit for $\boldA$, even when $\boldA$ may be non-separable (see \cite[Theorem 3.1.1]{Murphy.1990}).

\begin{lemma}\label{lemma:ComputableApproxUnitDirected}
If $a, b \in \Lambda$ are computable points of $\boldA^\#$ then $\Lambda$ contains a computable point $c$ of $\boldA^\#$ such that $a \leq c$ and $b \leq c$.  Moreover, an $\boldA^\#$-index of $c$ can be computed from $\boldA^\#$-indexes of $a$ and $b$.
\end{lemma}
\begin{proof}
Working in $\boldA_1$, let
\[c = (a(1_{\boldA_1}-a)^{-1}+b(1_{\boldA_1}-b)^{-1})(1_{\boldA_1}+a(1_{\boldA_1}-a)^{-1}+b(1_{\boldA_1}-b)^{-1})^{-1}.\]
It is shown on \cite[page 78]{Murphy.1990} that $c \in \Lambda$ and that $a \leq c$ and $b \leq c$.  Since $\norm{a} < 1$ we have $(1_{\boldA_1}-a)^{-1} = \sum_{n=0}^{\infty}a^n$, and likewise for $b$.  Since $a$ and $b$ are computable points of $\boldA^\#$, this implies that $c$ is as well.
\end{proof}

\begin{lemma}\label{lemma:ComputableApproxUnitOneStep}
Suppose that $a \in \Lambda$, $a$ is a computable point of $\boldA^\#$, and $n \in \mathbb{N}$.  Then $\Lambda$ contains a computable point $e$ of $\boldA^\#$ such that whenever $c \in \Lambda$ and $c \geq e$ we have $\norm{a-ca} < 2^{-n}$.  Moreover, an $\boldA^\#$-index of $e$ can be computed from $n$ and an $\boldA^\#$-index of $a$.
\end{lemma}

\begin{proof}
Define $h : [0, \infty) \to \mathbb{C}$ by
\[h(t) = \begin{cases}1 & t \geq 2^{-2n} \\ 2^{2n+1}t-1 & 2^{-(2n+1)} \leq t < 2^{-2n} \\ 0 & t < 2^{-(2n+1)}\end{cases}\]
Note that $h$ is continuous and $h(0) = 0$.  Let $g = (1-2^{-(2n+1)})h\vert_{\sigma(a) \cup \{0\}}$.  Set $e = g(a)$.  By Corollary \ref{corollary:ApplyingFuncCalc}, $e$ is a computable point of $\boldA^\#$, uniformly in $n$ and an $\boldA^\#$-index of $a$.

Since $h \geq 0$, it follows that $e$ is positive.  We have $\norm{e} = \norm{g} \leq 1-2^{-(2n+1)} < 1$, so $e \in \Lambda$.  Suppose that $c \in \Lambda$ and $c \geq e$.

We work in $\boldA_1$; to simplify notation, we write $1$ for $1_{\boldA_1}$.  We note that both the norm and order relations for elements of $\boldA$ remain unchanged by passing to $\boldA_1$, the latter by spectral permanence.  Thus $0 \leq c \leq 1$, and hence also $0 \leq 1-c \leq 1$.  On the other hand, since $c \geq e$, we also have $1-c \leq 1-e$.  By functional calculus applied to $1-c$ we thus have
\[(1-c)^2 \leq 1-c \leq 1-e.\]
Conjugation preserves the order relation between self-adjoint elements (see \cite[Theorem 2.2.5]{Murphy.1990}), so $a^*(1-c)^2a \leq a^*(1-e)a$, and since $a$ is self-adjoint this gives $a(1-c)^2a \leq a(1-e)a$.  By definition, $e = g(a)$, so $e$ and $a$ commute, giving $a(1-e)a = a^2(1-e)$.  Therefore:
\begin{eqnarray*}
\norm{a-ca}^2 & = & \norm{(1-c)a}^2 \\
&= & \norm{((1-c)a)^*((1-c)a)} \\
&= &\norm{a(1-c)^2a} \\
&\leq &\norm{a(1-e)a} \\
&= &\norm{a^2(1-e)} \\
&=& \sup_{t \in \sigma(a)}\abs{t^2(1-g(t))}
\end{eqnarray*}

Since $a \in \Lambda$ we have $\sigma(a) \subseteq [0, 1]$.  A direct calculation shows that for every $t \in [0, 1]$ we have $0 \leq t^2(1-g(t)) \leq 2^{-(2n+1)}< 2^{-2n}$, completing the proof.
\end{proof}

\begin{lemma}
There is a c.e. set $\Lambda_0$ of generated points of $\boldA^\#$ that is dense in $\Lambda$.
\end{lemma}

\begin{proof}
Let $\Lambda_0 = \{c^*c : c \text{ is a generated point of $\boldA^\#$ and $\norm{c} < 1$}\}$.  Then $\Lambda_0 \subseteq \Lambda$ and $\Lambda_0$ is c.e.  Suppose that $a \in \Lambda$ and $\epsilon > 0$.  Since $a$ is positive, there exists $b \in \boldA$ such that $a = b^*b$, and since $\norm{a} < 1$ we also have $\norm{b} < 1$.  Hence, 
there is a generated point $c$ of $\boldA^\#$ such that $\norm{c-b} < \epsilon(\norm{b}+1)^{-1}$ and so that $\norm{c} < 1$.  It follows that $\norm{a-c^*c} < \epsilon$.
\end{proof}

\begin{theorem}\label{thm:ApproximateUnit}
There is an effective approximate unit of $\boldA^\#$.
\end{theorem}

\begin{proof}
Fix a c.e. set $\Lambda_0 \subseteq \Lambda$ as in the previous lemma, 
and fix an effective enumeration $(d_j)_{j \in \mathbb{N}}$ of $\Lambda_0$.  

We recursively define a sequence $(u_n)_{n\in\mathbb{N}}$.  Suppose that $u_j$ has been defined for all $j < n$.  By Lemma \ref{lemma:ComputableApproxUnitOneStep}, for 
each $k \in \{0, \ldots, n\}$, 
we can compute an $\boldA^\#$-index of a point 
$e_{k,n} \in \Lambda$ so that $\norm{d_k-cd_k} < 2^{-n}$ 
whenever $c \in \Lambda$ and $c \geq e_{k,n}$. 
By Lemma \ref{lemma:ComputableApproxUnitDirected}, we can then compute a $\boldA^\#$-index of a $u_n \in \Lambda$ such that $u_n \geq e_{k,n}$ for all $k \leq n$ and so that $u_n \geq u_j$ for all $j < n$.  Observe that the construction ensures that if $n \geq \max\{k,j\}$ then $\norm{d_k-u_nd_k} < 2^{-j}$.  We will show that this sequence $(u_n)_{n\in\mathbb{N}}$ is an effective approximate unit.

Fix a generated point $a$ of $\boldA^\#$, and let $b = (\norm{a}+1)^{-1}a$.  By Corollary \ref{corollary:PositiveNegativeParts}, $b_r^+, b_r^-, b_i^+$, and $b_i^-$ are computable points of $\boldA^\#$.  Moreover, it is possible to compute an $\boldA^\#$-index of each of them from an $\boldA^\#$-index of $a$.

Fix $k \in \mathbb{N}$.  Compute $M \in \mathbb{N}$ such that $\norm{a}+1 < 2^M$.  Set 
\[k_0 = k + M + 4.\]
Compute $k(r, +) \in \mathbb{N}$ such that $\norm{b_r^+-d_{k(r,+)}} < 2^{-k_0}$, and likewise compute such $k(r, -)$, $k(i, +)$, and $k(i, -)$.  Let 
\[N_0 = \max\{k(r,+), k(r,-), k(i,+), k(i, -), k_0\}.\]

Consider a fixed $n \geq N_0$.  Since $\norm{u_n} < 1$ and $n \geq k(r, +)$, using the reverse triangle inequality and calculating in $\boldA_1$, we have
\begin{eqnarray*}
\abs{\norm{b_r^+-u_nb_r^+} - \norm{d_{k(r,+)} - u_nd_{k(r,+)}}} & \leq & \norm{(b_r^+-u_nb_r^+) - (d_{k(r,+)}-u_nd_{k(r,+)})} \\
&= &\norm{(1_{\boldA_1} - u_n)(b_r^+-d_{k(r,+)})} \\
&\leq & (\norm{1_{\boldA_1}} + \norm{u_n})\norm{b_r^+-d_{k(r,+)}} \\
&< &2\norm{b_r^+-d_{k(r,+)}} \\
&< &2\cdot 2^{-k_0}
\end{eqnarray*}
thus
\[\norm{b_r^+ - u_nb_r^+} < \norm{d_{k(r,+)} - u_nd_{k(r,+)}} +2\cdot 2^{-k_0}.\]
Also, since $n \geq \max\{k(r,+), k_0\}$ we have $\norm{d_{k(r,+)}-u_nd_{k(r,+)}} < 2^{-k_0}$, so we obtain
\[\norm{b_r^+ - u_nb_r^+} < 3\cdot2^{-k_0}.\]
Similar calculations produce the same estimate using $b_r^-$, $b_i^+$, and $b_i^-$. The decomposition $b = b_r^+ - b_r^- + ib_i^+ - ib_i^-$, the definition of $k_0$, and the choice of $M$ therefore imply:
\begin{eqnarray*}
\norm{b-u_nb} &< & 4\cdot3\cdot2^{-k_0} \\
&< & 2^{-k_0+4} \\
&= & 2^{-M}2^{-k} \\
&< & \frac{2^{-k}}{\norm{a}+1}
\end{eqnarray*}
Since $b = (\norm{a}+1)^{-1}a$, this implies
$\norm{a-u_na} < 2^{-k}$ as desired. 

Finally, to estimate $\norm{a-au_n}$, note that $b_r^+$ and $u_n$ are self-adjoint; thus, 
\[\norm{b_r^+ - b_r^+u_n} = \norm{(b_r^+-b_r^+u_n)^*} = \norm{(b_r^+)^* - u_n^*(b_r^+)^*} = \norm{b_r^+ - u_nb_r^+},\]
and similarly for $b_r^-$, $b_i^+$, and $b_i^-$.  These calculations produce the same upper bound for $\norm{b-bu_n}$ as they did for $\norm{b-u_nb}$, and it follows that $\norm{a - a u_n} < 2^{-k}$.
\end{proof}

\section{Spectral theory for compact operators}\label{section:CompactOperators}

We now turn our attention to compact operators on separable Hilbert spaces.  Throughout this section, let $\mathcal{H}$ be a separable Hilbert space and $\mathcal{H}^\#$ be a presentation of $\mathcal{H}$.  Let $(\rho_j)_{j \in \mathbb{N}}$ be an enumeration of the non-zero rational vectors of $\mathcal{H}^\#$.  For each non-zero $v \in \mathcal{H}$, let $P_v$ denote the orthogonal projection onto $v$.  We refer to $P_{\rho_j}$ as the \emph{$j$th rational rank $1$ projection}.  For convenience, we also let $P_0 = 0$.  We denote by $\mathcal{K}(\mathcal{H})$ the (non-unital) \cstar-algebra of compact operators on $\mathcal{H}$.

\subsection{Presentations of the compact operators}
Our first goal is to use $\mathcal{H}^\#$ to produce a presentation of $\mathcal{K}(\mathcal{H})$.

\begin{lemma}\label{lem:ProjectionNorm}
For all non-zero $v, w \in \mathcal{H}$,
\[\norm{P_v - P_w} \leq 2\norm{\norm{v}^{-1}v - \norm{w}^{-1}w}.\]
\end{lemma}
\begin{proof}
First, suppose that $v, w, u \in \mathcal{H}$ are unit vectors.  Then we have:
\begin{eqnarray*}
\norm{P_v(u) - P_w(u)} &=& \norm{\langle u,v \rangle v - \langle u, w \rangle w} \\
&=& \norm{\langle u, v \rangle v - \langle u, v \rangle w + \langle u, v \rangle w - \langle u, w \rangle w} \\
&\leq& \abs{\langle u, v \rangle}\norm{v-w} + \abs{\langle u,v \rangle - \langle u, w \rangle}\norm{w} \\
&=& \abs{\langle u, v\rangle}\norm{v-w} + \abs{\langle u, v-w \rangle} \\
&\leq& \norm{v-w} + \norm{v-w}\ \text{(Cauchy-Schwarz)}\\
&=& 2\norm{v-w}.
\end{eqnarray*}
Thus when $\norm{v} = \norm{w} = 1$ we have $\norm{P_v - P_w} \leq 2\norm{v-w}$.  Since $P_{\lambda v} = P_v$ for all $\lambda \in \mathbb{C} \setminus \{0\}$, and likewise for $w$, rescaling gives the desired conclusion.
\end{proof}

\begin{lemma}\label{lem:rationaldense}
The span of the set of rational rank $1$ projections is dense in $\mathcal{K}(\mathcal{H})$.    
\end{lemma}

\begin{proof}
By \cite[Theorems 2.4.5 and 2.4.6]{Murphy.1990}, the span of the set of all rank $1$ projections is dense in $\mathcal{K}(\mathcal{H})$.  It therefore suffices to show that every rank $1$ projection can be approximated by rational rank $1$ projections.  The non-zero rational vectors of $\mathcal{H}^{\#}$ are dense in $\mathcal{H}$, so Lemma \ref{lem:ProjectionNorm} shows that the rational rank $1$ projections are dense in the rank $1$ projections, and hence the span of the rational rank $1$ projections is dense in $\mathcal{K}(\mathcal{H})$.
\end{proof}

\begin{definition}
By $\mathcal{K}(\mathcal{H}^\#)$ we denote the presentation of $\mathcal{K}(\mathcal{H})$ whose $j$th special point is $P_{\rho_j}$.
\end{definition}

By Lemma \ref{lem:rationaldense}, $\mathcal{K}(\mathcal{H}^\#)$ is indeed a presentation of $\mathcal{K}(\mathcal{H})$.  The norm estimate given by Lemma \ref{lem:ProjectionNorm} immediately implies:

\begin{corollary}\label{cor:ComputePtProj}
If $v$ is a non-zero computable point of $\mathcal{H}^\#$ then $P_v$ is a computable point of $\mathcal{K}(\mathcal{H}^\#)$.  Moreover, from an $\calH^\#$-index for $v$ we can compute a $\mathcal{K}(\calH^\#)$-index for $P_v$.
\end{corollary}

We next aim to show that $\mathcal{K}(\mathcal{H}^\#)$ is a computable presentation whenever $\mathcal{H}^\#$ is computable.

\begin{lemma}\label{lem:IdentityMapLift}
Suppose that $\mathcal{H}^\dagger$ is a presentation of $\mathcal{H}$ such that the identity map is a computable map from $\mathcal{H}^\#$ to $\mathcal{H}^\dagger$.  Then the identity map is a computable map from $\mathcal{K}(\mathcal{H}^\#)$ to $\mathcal{K}(\mathcal{H}^\dagger)$.
\end{lemma}
\begin{proof} 
It suffices to show that from a generated point $\rho$ of $\mathcal{K}(\calH^\#)$ and $k \in \N$ it is possible to compute 
a generated point $\rho'$ of $\mathcal{K}(\calH^\dagger)$ so that $\norm{\rho - \rho'} < 2^{-k}$. 

Suppose $\rho = p(P_{\rho_0}, \ldots, P_{\rho_n})$ where $p$ is a rational $*$-polynomial with no constant term.  
From $p$, it is possible to compute a modulus of continuity $g$ for $p$ on the unit ball of $\C^{n+1}$; that is, a function $g : \mathbb{N} \to \mathbb{N}$ such that whenever $a_0, \ldots, a_n, b_0, \ldots, b_n \in \C$ all have modulus at most $1$ and $\max_j \norm{a_j - b_j} \leq 2^{-g(m)}$ we have $\norm{p(a_0, \ldots, a_n) - p(b_0, \ldots, b_n)} < 2^{-m}$. 

Let $\xi_j$ denote the $j$-th non-zero rational vector of $\calH^\dagger$.   Since the identity map from $\mathcal{H}^\#$ to $\mathcal{H}^\dagger$ is computable, each $\xi_j$ is a computable point of $\mathcal{H}^\#$. Thus, by Lemma \ref{lem:ProjectionNorm}, we can compute non-zero $\xi_{j_0}, \ldots, \xi_{j_n}$ so that 
$\norm{P_{\rho_s} - P_{\xi_{j_s}}} < 2^{-g(k)}$ for all $s \leq n$.  
Thus,
\[\norm{\rho - p(P_{\xi_{j_0}}, \ldots, P_{\xi_{j_n}})} < 2^{-k}.\]
\end{proof}

\begin{proposition}\label{prop:PresentationOfCompacts}
If $\mathcal{H}^\#$ is computable then so is $\mathcal{K}(\mathcal{H}^\#)$.
\end{proposition}
\begin{proof}
The proofs in the infinite-dimensional case and the finite-dimensional case are essentially identical, so we use the notation of the infinite-dimensional case.

Using Fact \ref{fact:ONB}, let $(e_n)_{n \in \mathbb{N}}$ be a computable orthonormal basis for $\mathcal{H}^\#$.  

Let $\mathcal{H}^\dagger$ be the presentation of $\mathcal{H}$ whose $n$th special point is $e_n$.  Since $(e_n)_{n\in\mathbb{N}}$ is a computable sequence of $\mathcal{H}^\#$, the identity map is a computable map from $\mathcal{H}^\#$ to $\mathcal{H}^\dagger$.  By Lemma \ref{lem:IdentityMapLift} the identity map is a computable map from $\mathcal{K}(\mathcal{H}^\#)$ to $\mathcal{K}(\mathcal{H}^\dagger)$ and thus is a computable *-isomorphism. Therefore, it suffices to show that $\mathcal{K}(\mathcal{H}^\dagger)$ is a computable presentation.

Let $\rho$ be a generated point of $\mathcal{K}(\mathcal{H}^\dagger)$.  For each $j$, let $\xi_j$ denote the $j$th non-zero rational vector of $\mathcal{H}^\dagger$.  Then we can write $\rho = p(P_{\xi_0}, \ldots, P_{\xi_m})$ for some rational $*$-polynomial $p$ and some $m$.  For each $j$, $\xi_j$ is a finite $\mathbb{Q}(i)$-linear combination of the basis vectors $e_0, e_1, \ldots$.  A code of $\xi_j$ provides us with $N_j$ such that $\xi_j \in \operatorname{span}(e_0, \ldots, e_{N_j})$, as well as the coefficients of $\xi_j$ as a linear combination of $e_0, \ldots, e_{N_j}$.  

Let $N = \max\{N_0, \ldots, N_m\}$ and let $E = \operatorname{span}(e_0, \ldots, e_N)$.  Then for $0 \leq j \leq m$, we have $P_{\xi_j}(E) \subseteq E$ and $P_{\xi_j}|_{E^\perp} = 0$.  These properties then extend to any monomial in $P_{\xi_0}, \ldots, P_{\xi_m}$, and hence (since $p$ has no constant term) we also have $\rho(E) \subseteq E$ and $\rho|_{E^\perp} = 0$.  Hence, $\norm{\rho} = \norm{\rho|_E}$.

For each $j$, a code of $\xi_j$ provides a vector $v_j = (a_0^{(j)}, \ldots, a_N^{(j)})^t \in \mathbb{Q}(i)^{N+1}$ such that $\xi_j = \sum_{k=0}^N a_k^{(j)}e_k$.  Let $A_j = \frac{1}{\norm{v_j}^2}v_jv_j^* = \frac{1}{v_j^*v_j}v_jv_j^*$.  Then $A_j$ is the matrix of ${P_{\xi_j}}|_E$ with respect to the basis $\{e_0, \ldots, e_N\}$ of $E$.  The matrix of $\rho|_E$ with respect to $\{e_0, \ldots, e_N\}$ is therefore $p(A_0, \ldots, A_m)$.  Thus 
\[\norm{\rho} = \norm{\rho|_E} = \norm{p(A_0, \ldots, A_m)},\] where the final norm is the operator norm on $M_{N+1}(\mathbb{C})$. The calculations we have described allow us to explicitly write the matrix $p(A_0, \ldots, A_m)$.  As each entry of $p(A_0, \ldots, A_m)$ is in $\mathbb{Q}(i)$, this matrix is a generated point of the standard presentation of $M_{N+1}(\mathbb{C})$, and since these standard presentations are computable, uniformly in $N$, we can thus compute $\norm{p(A_0, \ldots, A_m)}$.
\end{proof}

All separable Hilbert spaces are computably categorical (see \cite[Corollary 3.7]{BrattkaYoshikawa}), so between any two computable presentations of separable Hilbert spaces of the same dimension there is a computable isometric isomorphism.  We now show that computable isometric isomorphisms between Hilbert spaces lift to computable *-isomorphisms on the algebras of compact operators.

\begin{lemma}\label{lem:compConjugation}
Let $\mathcal{H}_0^\#$ and $\mathcal{H}_1^\#$ be computable presentations, and $\mathcal{H}_1$ and let $T$ be a computable isometric isomorphism from $\mathcal{H}_0^\#$ onto $\mathcal{H}_1^\#$.  
Then the map $\Phi_T : \mathcal{K}(\mathcal{H}_0) \to \mathcal{K}(\mathcal{H}_1)$ defined by $\Phi_T(S) = TST^{-1}$ is a computable *-isomorphism from $\mathcal{K}(\mathcal{H}_0^\#)$ to $\mathcal{K}(\mathcal{H}_1^\#)$.
\end{lemma}

\begin{proof}
Since $\Phi_T$ is a $*$-homomorphism, to show that $\Phi_T$ is computable it suffices to show that $\Phi_T$ is computable on the special points of $\mathcal{K}(\mathcal{H}_0^\#)$.  So we consider a special point, which has the form $P_\xi$ for some rational vector $\xi$ of $\mathcal{H}_0^\#$.  Then $\Phi_T(P_\xi) = P_{T(\xi)}$.  Since $T$ is computable, we can effectively approximate $T(\xi)$ arbitrarily well by rational vectors of $\mathcal{H}_1^\#$. Then it follows from Lemma \ref{lem:ProjectionNorm} that $P_{T(\xi)}$ is a computable point of $K(\mathcal{H}_1^\#)$.
\end{proof}

For each $n \in \mathbb{N}$, let $\ell^2_n$ be the $n$-dimensional Euclidean space, and let $\ell^2$ be the infinite-dimensional separable Hilbert space of square-summable sequences.  Each of these spaces has a standard computable presentation whose special points are given by the standard orthonormal bases.  Using these presentations, the following is immediate from Lemma \ref{lem:compConjugation} and the computable categoricity of Hilbert spaces.

\begin{proposition}
Suppose that $\mathcal{H}^\#$ is computable.
\begin{enumerate}
\item{If $\dim(\mathcal{H}) = n$, then $\mathcal{K}(\mathcal{H}^\#)$ is computably isomorphic to $\mathcal{K}(\ell_n^2)$.}
\item{If $\mathcal{H}$ is infinite-dimensional, then $\mathcal{K}(\mathcal{H}^\#)$ is computably isomorphic to $\mathcal{K}(\ell^2)$.}
\end{enumerate}
\end{proposition}

We have thus shown that computable presentations of $\mathcal{K}(\mathcal{H})$ arising from computable presentations of $\mathcal{H}$ are all computably isomorphic.  By \cite[Theorem 4.18]{UHFPaper}, all unital UHF algebras are computably categorical.  In the infinite-dimensional case, the compact operators $\mathcal{K}(\mathcal{H})$ form a \emph{non-unital} UHF algebra.  The methods from \cite{UHFPaper} do not apply to non-unital UHF algebras. Nevertheless, \cite[Theorem 4.18]{UHFPaper} and Lemma \ref{lem:compConjugation} together suggest the possibility of a positive answer to:

\begin{question}
Is $\mathcal{K}(\mathcal{H})$ computably categorical?    
\end{question}

\subsection{Computably compact operators}

Classically, compact operators on Hilbert spaces can be equivalently characterized as limits of finite-rank operators or as operators where the closure of the image of the closed unit ball is compact.  Our definition of $\mathcal{K}(\mathcal{H}^\#)$ is based on the former description.  The latter also has a natural computability-theoretic version.

\begin{definition}
Let $T$ be an operator on $\mathcal{H}$.  We say that $T$ is a \emph{computably image-compact operator of $\mathcal{H}^\#$} if $T$ is a computable operator of $\mathcal{H}^\#$ and $\overline{T[\overline{B}(0_\mathcal{H}; 1)]}$ is a computably compact set of $\mathcal{H}^\#$.
\end{definition}

A note on terminology is in order.  In \cite{BrattkaDillhage.2007}, the term ``computably compact operator" is used in the context of Banach spaces with the approximation property to mean a computable operator that can be computably approximated by finite-rank operators; in our setting, this notion corresponds to computable points of $\mathcal{K}(\mathcal{H}^\#)$.  The notion of computable image-compactness, defined above, makes sense even in the context of Banach spaces without the approximation property.  We remain in the context of operators on Hilbert spaces, where our next goal is to show that computably image-compact operators of $\mathcal{H}^\#$ are exactly the computable points of $\calK(\mathcal{H}^\#)$, thus giving a computability analog of the classical equivalence between the definitions of compact operators on Hilbert spaces.
\emph{We assume $\calH$ is infinite dimensional.}

\begin{lemma}\label{lem:FiniteRankApproximation1}
Suppose that $T$ is a computably image-compact operator of $\mathcal{H}^\#$, and fix $k \in \mathbb{N}$ and a computable orthonormal basis $(e_n)_{n\in\mathbb{N}}$ of $\mathcal{H}^\#$.  Then we can compute $K$, $n_1, \ldots, n_K$ such that $\norm{T - \sum_{i=1}^K P_{e_{n_i}} \circ T} < 2^{-k}$.
\end{lemma}

\begin{proof}
Since $(e_n)_{n \in \N}$ is a computable sequence of $\calH^\#$, it follows 
that $\{e_n : n \in \mathbb{N}\}$ is a c.e. closed set of $\calH^\#$.  Let $G = \{n \in \mathbb{N} : \operatorname{range}(P_{e_n} \circ T) \neq \{0_{\mathcal{H}}\}\}$.  Note that $G$ is a c.e. set.

Suppose $F \subseteq G$ and $F$ is finite.  We claim that $\norm{T - \sum_{n \in F}P_{e_n} \circ T}$ is a computable real number, uniformly in $F$.  To see this, let $C = \overline{T[\overline{B}(0_{\mathcal{H}}; 1)]}$.  Then
\[\norm{T - \sum_{n \in F}P_{e_n} \circ T} = \sup_{v \in C}\norm{v - \sum_{n \in F}P_{e_n}(v)}.\]
Since $e_n$ is a computable vector of $\mathcal{H}^\#$, uniformly in $n$, $\operatorname{Id}_{\mathcal{H}} - \sum_{n \in F}P_{e_n}$ is a computable operator of $\mathcal{H}^\#$, uniformly in $F$.  By hypothesis, $C$ is a computably compact set of $\mathcal{H}^\#$, so it follows that $\norm{T - \sum_{n \in F}P_{e_n} \circ T}$ is a computable real number, uniformly in $F$.

We now search for a finite set $F \subseteq G$ such that $\norm{T - \sum_{n \in F}P_{e_n} \circ T} < 2^{-k}$.  The previous paragraph shows that this search is effective and, by \cite[Corollary 4.5]{Conway.1990}, it terminates.
\end{proof}

\begin{lemma}\label{lemma:AdjointComputable}
Suppose that $T$ is a computably image-compact operator of $\mathcal{H}^\#$.  Then, the adjoint $T^*$ is also a computable operator of $\mathcal{H}^\#$.  An index for $T^*$ can be computed from an index for $T$ and an index for the computable compactness of $\overline{T[\overline{B}(0_{\mathcal{H}}; 1)]}$.
\end{lemma}

\begin{proof}
Since $T^*$ is linear and bounded, it suffices to show it is computable on the rational vectors of 
$\calH^\#$.  So, let $x$ be a rational vector of $\calH^\#$.  

Let $C = \overline{T[\overline{B}(0_{\mathcal{H}}; 1)]}$, and define the functional $\phi_x : \mathcal{H} \to \mathbb{C}$ by
\[\phi_x(v) = \langle T(v), x \rangle.\]
We have
\[\norm{\phi_x} = \sup_{y \in C}\abs{\langle y, x \rangle}.\]
As the inner product is a computable functional of $\calH^\#$, 
it follows that $\phi_x$ is a computable functional of $\calH^\#$ uniformly in $x$.  Furthermore, 
it follows from the computable compactness of $C$
that $\norm{\phi_x}$ is computable uniformly in $x$ (see \cite[Proposition 3.5]{DowneyMelnikov}).  

We have $\phi_x(v) = \langle T(v), x \rangle = \langle v, T^*(x) \rangle$.  Since $\phi_x$ and $\norm{\phi_x}$ are computable, by \cite[Corollary 4.4]{BrattkaYoshikawa} we get that $T^*(x)$ is computable.  Moreover, the proof of \cite[Corollary 4.4]{BrattkaDillhage.2007} is uniform, so $T^*(x)$ is computable uniformly in $x$.
\end{proof}

\begin{theorem}\label{theorem:EquivalentCompactness}
Suppose that $T$ is a bounded linear operator.  The following are equivalent:
\begin{enumerate}
    \item{$T$ is a computably image-compact operator of $\mathcal{H}^\#$.}
    \item{$T$ is a computable point of $\mathcal{K}(\mathcal{H}^\#)$.}
\end{enumerate}
\end{theorem}
\begin{proof}
First, suppose that $T$ is computably image-compact.  Using Fact \ref{fact:ONB} let $(e_n)_{n \in \mathbb{N}}$ be a computable orthonormal basis of $\mathcal{H}^\#$.  Fix $k \in \mathbb{N}$.  Using Lemma \ref{lem:FiniteRankApproximation1}, compute $K, n_1, \ldots, n_K$ such that
\[\norm{T - \sum_{j=1}^KP_{e_{n_j}}\circ T} < 2^{-k}.\]
For each $j$, and any $x, v \in \mathcal{H}$, we have the polarization identity
\begin{eqnarray*}
4\langle x, v\rangle e_{n_j} &=& \norm{e_{n_j}+v}^2P_{e_{n_j}+v}(x) - \norm{e_{n_j}-v}^2P_{e_{n_j}-v}(x) \\
&+& i\norm{e_{n_j}+iv}^2P_{e_{n_j}+iv}(x) -i\norm{e_{n_j}-iv}^2P_{e_{n_j}-iv}(x).
\end{eqnarray*}
We also have, for each $j$ and each $x \in \mathcal{H}$,
\[P_{e_{n_j}} \circ T(x) = \langle T(x), e_{n_j} \rangle e_{n_j} = \langle x, T^*(e_{n_j})\rangle e_{n_j}.\]
By setting $v = T^*(e_{n_j})$ in the polarization identity, we obtain 
\begin{eqnarray*}
P_{e_{n_j}} \circ T &=& \frac{\norm{e_{n_j}+T^*(e_{n_j})}^2}{4}P_{e_{n_j}+T^*(e_{n_j})} - \frac{\norm{e_{n_j}-T^*(e_{n_j})}^2}{4}P_{e_{n_j}-T^*(e_{n_j})} \\
&+ &i\frac{\norm{e_{n_j}+iT^*(e_{n_j})}^2}{4}P_{e_{n_j}+iT^*(e_{n_j})} - i\frac{\norm{e_{n_j}-iT^*(e_{n_j})}^2}{4}P_{e_{n_j}-iT^*(e_{n_j})}.
\end{eqnarray*}

By Lemma \ref{lemma:AdjointComputable}, $T^*$ is a computable operator, so each $T^*(e_t)$ is a computable vector, uniformly in $t$.  By Corollary \ref{cor:ComputePtProj}, we have thus written $P_{e_{n_j}} \circ T$ as a linear combination of computable points of $\mathcal{K}(\mathcal{H}^\#)$ with computable coefficients.  Letting $S_k = \sum_{j=1}^KP_{e_{n_j}}\circ T$, we have shown that $S_k$ is a computable point of $\mathcal{K}(\mathcal{H}^\#)$, uniformly in $k$, so $(S_k)_{k \in \mathbb{N}}$ is a computable sequence of $\mathcal{K}(\mathcal{H}^\#)$.  As $\norm{T - S_k} < 2^{-k}$ for all $k$, this completes the proof that $T$ is a computable point of $\mathcal{K}(\mathcal{H}^\#)$.

For the converse direction, fix $k \in \mathbb{N}$.  Compute a generated point $g = p(P_{\xi_0}, \ldots, P_{\xi_m})$ of $\mathcal{K}(\mathcal{H}^\#)$ so that $\norm{g-T} < 2^{-(k+1)}$.  Let $X = \operatorname{span}(\xi_0, \ldots, \xi_m)$, and let $C$ be the closed unit ball of $X$.  Then $g[\mathcal{H}] \subseteq X$ and $g|_{X^\perp} = 0$.  It follows that $g[C] = g[\overline{B}(0_{\mathcal{H}}; 1)]$.  

By \cite[Proposition 3.7]{McNicholl.2020}, $C$ is computably compact in $X$.  Since computable compactness is preserved by computable functions, $g[C]$ is computably compact in $X$.  It is thus possible to compute rational open balls $B(x_1; 2^{-(k+1)}), \ldots, B(x_r; 2^{-(k+1)})$ in $X$ so that 
\[g[C] \subseteq \bigcup_{j=1}^rB(x_j; 2^{-(k+1)})\]
and
\[B(x_j; 2^{-(k+1)}) \cap g[C] \neq \emptyset.\]  

Then, since $g[C] = g[\overline{B}(0_{\mathcal{H}}; 1)]$ and $\norm{T - g} < 2^{-(k+1)}$, this implies
\[\overline{T[\overline{B}(0_{\mathcal{H}}; 1)]} \subseteq \bigcup_{j=1}^rB(x_j; 2^{-k}),\] so $T$ is computably image-compact.
\end{proof}

Brattka and Dillhage, \cite[Corollary 19]{BrattkaDillhage.2005}, showed that (in our terminology) computable points of $\mathcal{K}(\mathcal{H}^\#)$ have computably compact spectra (this result also follows directly from the combination of our Corollary \ref{cor:SpctrmNtUntl} and Proposition \ref{prop:PresentationOfCompacts}).  We thus also have the following.

\begin{corollary}
If $T$ is a computably image-compact normal operator of $\mathcal{H}^\#$, then $\sigma(T)$ is computably compact.
\end{corollary}

We conclude with an effective version of the Spectral Theorem for compact normal operators.  To simplify notation, when $\lambda$ is an eigenvalue of an operator we let $P_\lambda$ denote the projection onto the corresponding eigenspace.  In light of Theorem \ref{theorem:EquivalentCompactness} and the existing terminology used in \cite{BrattkaDillhage.2007}, we use the term ``computably compact operator" instead of ``computably image-compact operator".

The following definition appears, using different terminology, in \cite[Definition 9.1]{Ackerman}.

\begin{definition}\label{def:EffDscrt}
Suppose $X \subseteq \C$.  Then, $X$ is \emph{effectively discrete} if $X$ is covered by a c.e. set $\mathcal{S}$ of rational open disks each of which 
contains exactly one point of $X$.  An index for such an $\mathcal{S}$ is an \emph{index of effective discreteness} for $X$.
\end{definition}

Recall that the canonical index of a finite $E \subseteq \N$ is $\sum_{n \in E} 2^{n}$. 
If $F \subseteq \C$ is finite, then an \emph{index} of $F$ is a canonical index of a finite set $E \subseteq \N$ so that 
each $e \in E$ indexes a point in $F$ and each point in $F$ is indexed by exactly one number in $E$.  

\begin{theorem}\label{thm:Spectral}
Suppose that $T$ is a computably compact normal operator of $\mathcal{H}^\#$.  Then:
\begin{enumerate}
\item{Every eigenvalue of $T$ is computable.}
\item{If $\lambda$ is a non-zero eigenvalue of $T$, then $P_\lambda$ is a computable operator of $\calH^\#$.  Furthermore, if we have an index of effective discreteness for $\sigma(T)\setminus\{0\}$, then an $\calH^\#$-index of $P_\lambda$ can be computed from an index of $\lambda$.}
\item{If we have an index of effective discreteness for $\sigma(T)\setminus\{0\}$, then from $k \in \mathbb{N}$ it is possible to compute an index of a finite set $F$ of non-zero eigenvalues of $T$ so that $\norm{T - \sum_{\lambda \in F}\lambda P_\lambda} < 2^{-k}$.}
\end{enumerate}
\end{theorem}

\begin{proof}
(1): Suppose that $\lambda$ is an eigenvalue of $T$.  If $\lambda = 0$ then $\lambda$ is computable.  Otherwise, $\lambda$ is an isolated point of the computably compact set $\sigma(T)$ (see \cite[Theorem 1.4.11]{Murphy.1990}), and is therefore computable.

(2): Since $\lambda$ is isolated in $\sigma(T) \cup \{0\}$, the characteristic function $\chi_{\lambda}$ of $\{\lambda\}$ is a computable function from $\sigma(T) \cup \{0\}$ to $\mathbb{C}$, so the first part of the result follows from Corollary \ref{corollary:ApplyingFuncCalc}  (using the presentation $\mathcal{K}(\mathcal{H}^\#)$ and the fact that $P_\lambda = \chi_{\lambda}(T)$).  Now, assume we have an index for a c.e. set $\calS$ of rational open disks covering $\sigma(T) \setminus \{0\}$, each of which contains exactly one eigenvalue of $T$.  
Suppose we are given an index $e$ of $\lambda \in \sigma(T)\setminus\{0\}$.  From this index, 
it is possible to compute a $D \in \mathcal{S}$ so that $\lambda \in D$ and so that $0 \not \in \overline{D}$; since $\lambda$ is computable, we can compute a rational open ball $D'$ with the same centre as $D$ and with a small enough radius so that such that $\lambda \in D'$ and $\overline{D'} \subseteq D$.  It now follows that $\chi_\lambda |_{\sigma(T)}$ is 
computable via the following procedure: given a name of a point $z \in \sigma(T)$, wait for a rational open disk $D_0$ so that $z \in D_0$ and 
$D_0 \subseteq D'$ or $D_0 \cap \overline{D'} = \emptyset$; in the former case output $1$ and in the latter output $0$.  An index of this procedure 
can be computed from $e$.

(3): Being computably compact, $\sigma(T)$ is in particular c.e. closed, and therefore contains a computable dense sequence of points $(\lambda_n)_{n \in \mathbb{N}}$ (see \cite[Lemma 3.27]{DowneyMelnikov}).  Since every point of $\sigma(T)$ (except possibly $0$) is isolated, every non-zero eigenvalue of $T$ occurs (possibly repeatedly) in this dense sequence.  Given an index of a finite set $E \subseteq \N$ we can compute $\norm{T - \sum_{n \in E}\lambda_n P_{\lambda_n}}$ because each $P_\lambda$ is computable, uniformly in an index of $\lambda$.  
Hence, it is possible to effectively search for a finite $E \subseteq \N$ so that $\norm{T - \sum_{n \in E}\lambda_n P_{\lambda_n}} < 2^{-k}$
and so that $\lambda_m \neq \lambda_n$ whenever $m,n \in E$ are distinct. 
Since the classical Spectral Theorem tells us that $T = \sum_{n \in \mathbb{N}}\lambda P_\lambda$, this search terminates.
Set $F = \{\lambda_m\ :\ m \in E\}$.  From $E$, it is possible to compute an index of $F$. 
\end{proof}


\bibliographystyle{amsplain}
\bibliography{paperbib}

\end{document}